\documentclass{article}

\usepackage[T1]{fontenc}
\usepackage[english]{babel}
\usepackage[dvips]{epsfig}
\usepackage{amsmath}
\usepackage{amsthm, amssymb, tabularx,eepic,epic}

\begin{document}
\theoremstyle{plain}
\newtheorem{thm}{Theorem}
\newtheorem{lemma}[thm]{Lemma}
\newtheorem{cor}[thm]{Corollary}

\theoremstyle{definition}
\newtheorem{defin}[thm]{Definition}
\newtheorem{ex}[thm]{Example}
\newtheorem{problem}[thm]{Problem}
\newtheorem*{rem}{Remark}
\newtheorem*{ack}{Acknowledgement}

\renewcommand\P{\operatorname{\mathbb P{}}}
\newcommand{\diam}{\mbox{diam}}

\newenvironment{romenumerate}{\begin{enumerate}
 \renewcommand{\labelenumi}{\textup{(\roman{enumi})}}%
 \renewcommand{\theenumi}{\textup{(\roman{enumi})}}%
 }{\end{enumerate}}

\newcommand{\refT}[1]{Theorem~\ref{#1}}
\newcommand{\refC}[1]{Corollary~\ref{#1}}
\newcommand{\refL}[1]{Lemma~\ref{#1}}
\newcommand{\refE}[1]{Example~\ref{#1}}
\newcommand{\refS}[1]{Section~\ref{#1}}

\newcommand\marginal[1]{\marginpar{\raggedright\parindent=0pt\tiny #1}}
\newcommand\SJ{\marginal{SJ}}
\newcommand\XXX{XXX \marginal{XXX}}
\newcommand\REM[1]{{\raggedright\texttt{[#1]}\par\marginal{XXX}}}
\newcommand{\kolla}{\marginal{KOLLA!}}

\newcommand\set[1]{\ensuremath{\{#1\}}}
\newcommand\bigset[1]{\ensuremath{\bigl\{#1\bigr\}}}
\newcommand\Bigset[1]{\ensuremath{\Bigl\{#1\Bigr\}}}
\newcommand\biggset[1]{\ensuremath{\biggl\{#1\biggr\}}}
\newcommand\xpar[1]{(#1)}
\newcommand\bigpar[1]{\bigl(#1\bigr)}
\newcommand\Bigpar[1]{\Bigl(#1\Bigr)}
\newcommand\biggpar[1]{\biggl(#1\biggr)}
\newcommand\lrpar[1]{\left(#1\right)}
\newcommand\bigsqpar[1]{\bigl[#1\bigr]}
\newcommand\Bigsqpar[1]{\Bigl[#1\Bigr]}
\newcommand\biggsqpar[1]{\biggl[#1\biggr]}
\newcommand\lrsqpar[1]{\left[#1\right]}
\newcommand\bigabs[1]{\bigl|#1\bigr|}
\newcommand\Bigabs[1]{\Bigl|#1\Bigr|}
\newcommand\biggabs[1]{\biggl|#1\biggr|}
\newcommand\lrabs[1]{\left|#1\right|}

\newcommand{\lex}{{\le}\,}
\newcommand{\NN}{\Xi}
\newcommand{\Gr}{\mathfrak{Gr}}
\newcommand{\cupx}{\cup}
\newcommand{\Nx}{N[x]}
\newcommand{\Ny}{N[y]}
\newcommand{\Nz}{N[z]}
\newcommand{\Nw}{N[w]}
\newcommand{\Nxi}{N[x_i]}
\newcommand{\Nxj}{N[x_j]}
\newcommand{\gd}{\delta}
\newcommand{\gl}{\lambda}
\def\symdiff{\bigtriangleup}
\newcommand{\SRG}{\ensuremath{\text{-SRG}}}
\newcommand{\RSHCD}{\ensuremath{\text{RSHCD}}}

\title{Graphs where every $k$-subset of vertices is an identifying set}
\author{Sylvain Gravier\thanks{CNRS - Institut Fourier, ERT Maths \`a Modeler, 100 rue des Maths BP 74, 38402 Saint Martin d'Hères, France, Sylvain.Gravier@ujf-grenoble.fr}\and
Svante Janson\thanks{Uppsala University, Department of Mathematics
P.O. Box 480 S-751 06 Uppsala, Sweden, svante.janson@math.uu.se}\and
Tero Laihonen \thanks{Department of Mathematics,
University of Turku, 20014 Turku, Finland, terolai@utu.fi. Research supported by the Academy of Finland under grant 111940.} \and Sanna Ranto\thanks{Department of Mathematics, University of Turku, 20014 Turku, Finland,
samano@utu.fi. Research supported by the Academy of Finland under grant 111940.}}

\date{\today} 
\maketitle
\begin{abstract} 
Let $G=(V,E)$ be an undirected graph without loops and multiple edges.
A subset $C\subseteq V$ is called \emph{identifying} if for every
vertex $x\in V$ 
the intersection of $C$ and the closed neighbourhood
of $x$ is nonempty, and these intersections are different for
different vertices $x$. 

Let $k$ be a positive integer. We will consider graphs where \emph{every} $k$-subset is identifying. We prove that for every  $k>1$ the maximal order of such a graph is at most $2k-2.$ Constructions attaining the maximal order are given for infinitely many values of $k.$

The corresponding problem of $k$-subsets identifying any at most $\ell$ vertices is considered as well.
\end{abstract}

\section{Introduction}

Karpovsky \emph{et al.} introduced identifying sets in \cite{kcl} for locating faulty procesors in multiprocessor systems. Since then identifying sets have been considered in many different graphs (see numerous references in \cite{lowww}) and they find their motivations, for example, in sensor networks and enviromental monitoring \cite{LT:IdcodeCovering}. For recent developments see for instance \cite{Auger, ACHHL}.

Let $G=(V,E)$ be a simple undirected graph where $V$ is the set of vertices and $E$ is
the set of edges. The adjacency between vertices $x$ and $y$ is denoted by $x\sim y$, and an edge between $x$ and $y$ is denoted by $\{x,y\}$ or $xy$. Suppose $x,y\in V$. The \emph{(graphical)
distance} between $x$ and $y$ is the shortest path between these
vertices and it is denoted by $d(x,y)$. If there is no such path, then $d(x,y)=\infty$. We denote by $N(x)$ the set
of vertices adjacent to $x$ (\emph{neighbourhood}) and the
\emph{closed neighbourhood} of a vertex $x$ is $N[x]=\{x\}\cup
N(x)$. The closed neighbourhood within radius $r$ centered at $x$ is
denoted by $N_r[x]=\{y\in V\mid d(x,y)\le r\}$. We denote further
$S_r(x)=\{y\in V\mid d(x,y)=r\}$. Moreover, for $X\subseteq V$,
$N_r[X]=\cup_{x\in X}N_r[x]$.  For $C\subseteq V$, 
$X\subseteq V$, and $x\in V$ we denote
\begin{gather*}
I_r(C;x)=I_r(x)=N_r[x]\cap C,
\\
I_r(C;X)=I_r(X)=N_r[X]\cap C
=\bigcup_{x\in X} I_r(C;x).
\end{gather*}
If $r=1$, we drop it from the notations.
When necessary, we add a subscript $G$.
We also write, for example, $N[x,y]$ and $I(C;x,y)$ for $N[\set{x,y}]$
and $I(C;\set{x,y})$.
The \emph{symmetric
difference} of two sets is
\[A\symdiff B=(A\setminus B)\cup (B\setminus A).\]
The cardinality of a set $X$ is denoted by $|X|$; we will also write
$|G|$ for the order $|V|$ of a graph $G=(V,E)$.
The \emph{degree}
of a vertex $x$ is $\deg(x)=|N(x)|$. Moreover, $\delta_G=\delta=\min_{x\in
V}\deg(x)$ and $\Delta_G=\Delta=\max_{x\in V}\deg (x)$. 
The \emph{diameter} of a graph $G=(V,E)$ is 
$\diam(G)=\max\{d(x,y)\mid x,y\in V\}$. 

We say that a vertex $x\in V$ \emph{dominates} a vertex $y\in V$ if
and only if $y\in N[x]$. As well we can say that a vertex $y$ is
\emph{dominated} by $x$ (or vice versa). A subset $C$ of vertices
$V$ is called a \emph{dominating set} (or \emph{dominating}) if
$\cup_{c\in C}N[c]=V$.

\begin{defin}
A subset $C$ of vertices of a graph $G=(V,E)$ is called
$(r,\lex\ell)$-\emph{identifying} (or an
$(r,\lex\ell)$-\emph{identifying set}) if for all $X,Y\subseteq V$
with
$|X|\le\ell$, $|Y|\le \ell$, $X\ne Y$ we have
\[I_r(C;X)\ne I_r(C;Y).\]

If $r=1$ and $\ell=1$, then we speak about an \emph{identifying
set}.
\end{defin}

The idea behind identification is that we can uniquely determine the subset $X$ of vertices of a graph $G=(V,E)$ by knowing only $I_r(C;X)$ --- provided that $|X|\le \ell$ and $C\subseteq V$ is an $(r,\le \ell)$-identifying set.

\begin{defin}
Let, for $n\ge k\ge 1$ and $\ell\ge1$, $\Gr(n,k,\ell)$ be the set of graphs 
on $n$ vertices such that every $k$-element
set of vertices is $(1,\lex\ell)$-identifying.
Moreover, we denote $\Gr(n,k,1)=\Gr(n,k)$
and $\Gr(k)=\bigcup_{n\ge k} \Gr(n,k)$.
\end{defin}

\begin{ex}\label{exa}
  \begin{romenumerate}
\item\label{exaE}
For every $\ell\ge1$, an empty graph $E_n=(\{1,\ldots,n\},\emptyset)$
belongs to $\Gr(n,k,\ell)$ if and only if $k=n$.

\item\label{exaC}
A cycle $C_n$  ($n\ge 4$) belongs to $\Gr(n,k)$ if and only if
$n-1\le k\le n$. A cycle $C_n$ with $n\ge 7$ is in $\Gr(n,n,2)$.

\item\label{exaP}
A path $P_n$ of $n$ vertices ($n\ge 3$) belongs to $\Gr(n,k)$ if
and only if $k=n$.

\item\label{exaKK}
A complete bipartite graph $K_{n,m}$ $(n+m\ge 4)$ is in $\Gr(n+m,k)$
if and only $n+m-1\le k\le n+m$.

\item\label{exaS}
In particular, 
a star $S_n=K_{1,n-1}$ 
($n\ge4$) is in $\Gr(n,k)$ if and only if $n-1\le k\le n$.

\item\label{exaK}
The complete graph $K_n$ ($n\ge2$) is not in $\Gr(n,k)$ for any $k$.
  \end{romenumerate}
\end{ex}

We are interested in the maximum number $n$ of vertices which can be
reached by a given $k$. 
We study mainly the case $\ell=1$ and define
\begin{equation}
  \label{NN}
\NN(k)=\max\set{n:\Gr(n,k)\neq\emptyset}.
\end{equation}
Conversely, the question is for a given
graph on $n$ vertices what is the smallest number $k$ such that
every $k$-subset of vertices is an identifying set (or a $(1,\lex
\ell)$-identifying set). 
(Note that even if we take $k=n$, there are graphs on $n$ vertices that
do not belong to $\Gr(n,n)$, for example the complete graph $K_n$, $n\ge2$.)
The relation $n/k$ is called the \emph{rate}.

In particular, we are interested in the asymptotics as $k\to\infty$. 
Combining \refT{kbound} and \refC{Clower}, we obtain the following,
which in particular shows that the rate is always less than 2.

\begin{thm}\label{T2}
$\NN(k)\le 2k-2$ for all $k\ge2$, and
$ \lim_{k\to\infty}\frac{\NN(k)}{k}=2$.
\end{thm}

We will see in \refS{sec:strong} that 
$\NN(k)=2k-2$ for infinitely many $k$. 

\begin{rem}
We consider in this paper the set $\Gr(n,k,\ell)$ only for
$(1,\lex\ell)$-identifying 
sets, i.e.\ with radius $r=1$, because increasing the radius does not
increase the maximum number of vertices for given $k$ and $\ell$.
Namely, if $G$ is a graph
such that every $k$-subset of vertices is $(r,\lex \ell)$-identifying
for a fixed $r\ge 2$, then the power graph of $G$, where every pair of
vertices with distance at most $r$ in $G$ are joined by an edge, belongs to
$\Gr(n,k,\ell)$. (However, the existence of a graph $G$ in
$\Gr(n,k,\ell)$ does not imply that every $k$-subset of vertices in $G$
is $(r,\lex \ell)$-identifying for $r\ge 2$.)  
\end{rem}

\begin{rem}
The similar question about graphs where every $k$-subset of vertices would
be a dominating set is easy. Namely, every vertex of a
complete graph with $n$ vertices forms alone a dominating set for
all $n$, so for this problem, $n$ can be arbitrary, even for $k=1$.
\end{rem}


We give some basic results in \refS{Sbasic}, including our
first upper bound on $\NN(k)$.
A better bound, based on a relation with error-correcting codes,
is given in \refS{Supper}, but we first
study small $k$ in \refS{Ssmall}, where we give a complete
description of the sets $\Gr(k)$ for $k\le 4$ and find $\NN(k)$ for
$k\le 6$.
We consider strongly regular graphs and some modifications of them in
\refS{sec:strong}; this provides us with examples (e.g., Paley graphs)
that attain or almost attain the upper bound in \refT{T2}.
In \refS{Ssmaller} we consider the probability that a random subset of
$s$ vertices in a graph $G\in\Gr(n,k)$ is identifying (for $s<k$); in
particular, this yields results on the size of the smallest
identifying set.
In \refS{Sell} we give some results for the case $\ell\ge2$.

\section{Some basic results}\label{Sbasic}

We begin with some simple consequences of the definition.

\begin{lemma}\label{L1}
  \begin{romenumerate}
	\item\label{L1a}
If $G\in \Gr(n,k,\ell)$, then $G\in \Gr(n,k',\ell')$ 
whenever $k\le k'\le n$ and $1\le\ell'\le\ell$. 
\item\label{L1b}
If $G=(V,E)\in \Gr(n,k,\ell)$, then every induced 
subgraph $G[A]$, where $A\subseteq V$, of order $|A|=m\ge k$ belongs to
$\Gr(m,k,\ell)$. 
\item\label{L1c}
If $\Gr(n,k)=\emptyset$, then $\Gr(n',k)=\emptyset$ for all $n'\ge n$.
  \end{romenumerate}
\end{lemma}
\begin{proof}
  Parts \ref{L1a} and \ref{L1b} are straightforward to verify. For \ref{L1c}, note
  that any subset of $n$ vertices of a graph in $\Gr(n',k)$
would induce a graph in $\Gr(n,k)$ by \ref{L1b}.
\end{proof}

\begin{lemma}\label{Lcomp}
If $G$ has connected components $G_i$, $i=1,\dots,m$,
with $|G|=n$ and $|G_i|=n_i$, 
then $G\in\Gr(n,k,\ell)$ if and only if $G_i\in\Gr(n_i,k+n_i-n,\ell)$
for every $i$.
In other words, $G_i\in\Gr(n_i,k_i,\ell)$ with $n_i-k_i=n-k$.
\end{lemma}
\begin{proof}
  Every $k$-set of vertices contains at least $k_i=k-(n-n_i)$ vertices
  from $G_i$. Conversely, every $k_i$-set of vertices of $G_i$ can be
  extended to a $k$-set of vertices of $G$ by adding all vertices in
  the other components. The result follows easily.
\end{proof}

A graph
$G$ belongs to $\Gr(n,k,\ell)$ if and only if every $k$-subset
intersects every symmetric difference of the neighbourhoods of two
sets that are of size at most $\ell$. 
Equivalently, $G\in \Gr(n,k,\ell)$ if and only if the
complement of every such  symmetric difference of two neighbourhoods 
contains less than $k$ vertices. We state this as a theorem.

\begin{thm}\label{karakterisointi}
Let $G=(V,E)$ and $|V|=n$.
$G$ belongs to $\Gr(n,k,\ell)$ if and only if 
\begin{equation}\label{eqkarak}
n-\min_{\substack{X,Y\subseteq V\\X\neq Y\\|X|,|Y|\le \ell}}
\{|N[X] \symdiff N[Y]|\}\le k-1.
\end{equation}
\end{thm}

Now take $\ell=1$, and consider $\Gr(n,k)$. The characterization 
in Theorem \ref{karakterisointi} can be written as follows, since $X$
and $Y$ either are empty or singletons. 

\begin{cor}\label{Cekaehto}
Let $G=(V,E)$ and $|V|=n$.
$G$ belongs to $\Gr(n,k)$ if and only if 
\begin{romenumerate}
\item$\delta_G\ge n-k$, and \label{ekaehto}
\item$\max_{x,y\in V,\;x\neq y}\{|N[x]\cap N[y]|+
 |V\setminus(N[x]\cup N[y])|\}\le k-1$.\label{tokaehto} 
\end{romenumerate}

In particular, if $G\in \Gr(n,k)$ then
every vertex is dominated by every choice of a
$k$-subset, and
for all distinct $x,y\in V$ we have
$ |N[x]\cap N[y]|\le k-1$.
\end{cor}

\begin{ex}
  \label{Ecube}
Let $G$ be the 3-dimensional cube, with 8 vertices.
Then $|N[x]|=4$ for every vertex $x$, and $|N[x]\symdiff N[y]|$ is
$4$ when $d(x,y)=1$, 4 when $d(x,y)=2$, and 8 when $d(x,y)=3$. Hence,
\refT{karakterisointi} shows that $G\in\Gr(8,5)$.
\end{ex}

\begin{lemma}
\label{Ladd1}
Let $G_0=(V_0,E_0)\in\Gr(n_0,k_0)$
and let $G=(V_0\cup\{a\},E_0\cup\{\{a,x\}\mid x\in V_0\})$ for a new vertex
$a\notin V_0$. In  words, we add a vertex and connect it to all other
vertices.
Then $G\in\Gr(n_0+1,k_0+1)$ if (and only if) $|N_{G_0}[x]|\le k_0-1$
for every $x\in V_0$, or,
equivalently, $\Delta_{G_0}\le k_0-2$.
\end{lemma}
\begin{proof}
  An immediate consequence of 
\refT{karakterisointi} (or \refC{Cekaehto}).
\end{proof}

\begin{ex}
  \label{Ebcc}
If $G_0$ is the 3-dimensional cube in \refE{Ecube}, which belongs to
$\Gr(8,5)$ and is regular with degree $3=5-2$, then \refL{Ladd1}
yields a graph $G\in\Gr(9,6)$. $G$ can be regarded as a cube with centre.
\end{ex}

Suppose $G=(V,E)$ belongs to $\Gr(n,k)$. 
Corollary \ref{Cekaehto}\ref{ekaehto} implies that for all $x\in V$,
$n-|N[x]|\le k-1$. On the other hand, 
\refL{Ladd1} shows that
there is not a positive lower bound for
$n-|N[x]|$, since the graph $G=(V,E)$ constructed there has a vertex $a$
such that $N[a]=V$. Arbitrarily large graphs $G_0$ satisfying the
conditions in \refL{Ladd1} are, for example, given by the Paley graphs
$P(q)$,
see Section~\ref{sec:strong}.  

We now easily obtain our first upper bound (which will be improved later) on
the order of a graph such that every $k$-vertex set is identifying.

\begin{thm}\label{Tfirstupperbound}
  If $k\ge2$ and $n>3k-3$, then there is no graph in $\Gr(n,k)$.
In other words, $\NN(k)\le 3k-3$ when $k\ge2$.
\end{thm}

\begin{proof}
Suppose $G\in \Gr(n,k)$ with $n\ge2$.
  Pick two distinct vertices $x$ and $y$.
By Corollary \ref{Cekaehto}\ref{ekaehto},
$|N[x]|,|N[y]|\ge n-k+1$ and thus 
\begin{equation*}
|N[x]\symdiff N[y]|\le 
|V \setminus N[x]|  + |V \setminus N[y]| 
\le k-1 + k-1=2k-2.
\end{equation*}
Consequently, 
Theorem~\ref{karakterisointi}
yields $n\le 2k-2+k-1=3k-3$. 
\end{proof}

As a corollary, $\Gr(k)$ is a finite set of graphs for every $k$.

\section{Small $k$}\label{Ssmall}

\begin{ex}\label{exgr1}
For $k=1$,   it is easily seen that $\Gr(n,1)=\emptyset$ for $n\ge2$,
and thus $\Gr(1)=\set{K_1}$ and $\NN(1)=1$.
\end{ex}

\begin{ex}\label{exgr2}
Let $k=2$.
If $G\in \Gr(2)$, then $G$ cannot contain any edge $xy$, since then
$N[x]\cap\set{x,y}=\set{x,y}=N[y]\cap\set{x,y}$, so \set{x,y} does not
separate $\set{x}$ and $\set{y}$.
Consequently, $G$ has to be an empty graph $E_n$, and then
$\delta_G=0$ and Corollary \ref{Cekaehto}\ref{ekaehto} 
(or Example \ref{exa}\ref{exaE}) shows that
$n=k=2$.
Thus $\Gr(2)=\set{E_2}$ and $\NN(2)=2$.
\end{ex}

\begin{ex}\label{exgr3}
Let $k=3$.
First, assume $n=|G|=3$. There are only four graphs $G$ with $|G|=3$,
and it is easily checked that $E_3,P_3\in\Gr(3,3)$ 
(\refE{exa}\ref{exaE}\ref{exaP}), while $C_3=K_3\notin\Gr(3,3)$
(\refE{exa}\ref{exaK}) and a disjoint union $K_1\cupx K_2\notin\Gr(3,3)$,
for example by \refL{Lcomp} since $K_2\notin\Gr(2,2)$.
Hence $\Gr(3,3)=\set{E_3,P_3}$.

Next, assume $n\ge4$. Since there are no graphs in $\Gr(n_1,k_1)$ if
$n_1>k_1$ and $k_1\le2$, it follows from \refL{Lcomp} that there are
no disconnected graphs in $\Gr(n,3)$ for $n\ge4$.
Furthermore, if $G\in\Gr(n,3)$, then  every induced
subgraph with 3 vertices is in $\Gr(3,3)$ and is thus $E_3$ or $P_3$;
in particular, $G$ contains no triangle.

If $G\in\Gr(4,3)$, 
it follows easily that $G$ must be $C_4$ or $S_4$, and indeed these
belong to $\Gr(4,3)$ by \refE{exa}\ref{exaC}\ref{exaS}. Hence
$\Gr(4,3)=\set{C_4,S_4}$. 

Next, assume $G\in\Gr(5,3)$. Then every induced subgraph with 4
vertices is in $\Gr(4,3)$ and is thus $C_4$ or $S_4$. Moreover, by
\refC{Cekaehto}, $\delta_G\ge5-3=2$. 
However, if we add a vertex to $C_4$ or $S_4$ 
such that the degree condition $\delta_G\ge2$ is satisfied and we do not
create a triangle we get $K_{2,3}$ -- a complete bipartite graph, and
we know already $K_{2,3}\not\in \Gr(5,3)$
(\refE{exa}\ref{exaKK}). Consequently $\Gr(5,3)=\emptyset$, and thus
$\Gr(n,3)=\emptyset$ for all $n\ge5$ by \refL{L1}\ref{L1c}.

Consequently, $\Gr(3)=\Gr(3,3)\cup\Gr(4,3)=\set{E_3,P_3,S_4,C_4}$ and
$\NN(3)=4$. 
\end{ex}

\begin{ex}\label{exgr4}
Let $k=4$.
First, it follows easily from \refL{Lcomp} and the descriptions of
$\Gr(j)$ for $j\le3$ above that the only disconnected graphs in
$\Gr(4)$ are $E_4$ and the disjoint union $P_3\cupx K_1$;
in particular, every graph in $\Gr(n,4)$ with $n\ge5$ is connected.

Next, if $G\in \Gr(n,4)$,
there cannot be a triangle in
 $G$ because otherwise if a 4-subset includes the vertices of a
 triangle, one more vertex cannot separate the vertices of the
 triangle from each other. 
(Cf.\ \refL{Ltriangleupperbound}.)

For $n=4$, the only connected graphs of order 4 that do not contain a
triangle are $C_4$, $P_4$ and $S_4$, and these
belong to $\Gr(4,4)$ by \refE{exa}\ref{exaC}\ref{exaP}\ref{exaS}. Hence
$\Gr(4,4)=\set{C_4,P_4,S_4,E_4,P_3\cupx K_1}$.

Now assume that $G\in\Gr(n,4)$ with $n\ge 5$.

(i) Suppose first that a graph $K_1\cupx K_2=(\{x,y,z\},\{\{x,y\}\})$ 
is an induced subgraph of $G$. Then all the other vertices of $G$ are
adjacent to either $x$ or $y$ but not both, since otherwise there
would be an induced triangle or an induced $E_2\cupx K_2$ or $K_2\cupx
K_2$, and these do not belong to $\Gr(4,4)$. Let
$A=N(x)\setminus \{y\}$ and $B=N(y)\setminus \{x\}$, so we have a
partition of the vertex set as $\set{x,y,z}\cup A \cup B$. There can be
further edges between $A$ and $B$, $z$ and $A$, $z$ and $B$ but not
inside $A$ and $B$. Let $A=A_0\cup A_1$ and $B=B_0\cup B_1$, where
$A_1=\{a\in A\mid a\sim z\}$, $A_0=A\setminus A_1$ and $B_1=\{b\in
B\mid b\sim z\}$, $B_0=B\setminus B_1$. 
If
$a\in A_0$ and
$b\in B$, then the 4-subset $\{a,b,x,z\}$ does not distinguish $a$ and $x$
unless $a\sim b$. Similarly, if $a\in A$ and $b\in B_0$, then $a\sim
b$. On the other hand, if $a\in A_1$ 
and $b\in B_1$, then $a\not\sim b$, since otherwise $abz$ would be a
triangle. 
Thus, we have, where one or more of the sets $A_0,A_1,B_0,B_1$ might
be empty,
\begin{center}
\begin{tabular}{ccc}
\includegraphics[height=2cm]{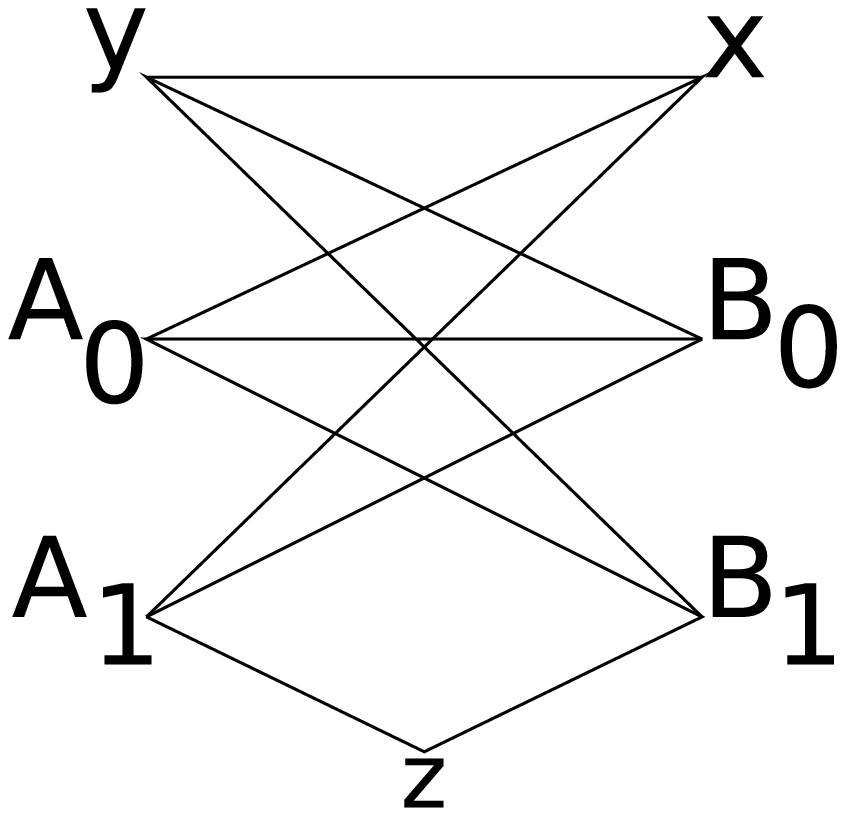}
& $=$ &
\includegraphics[width=2.5cm]{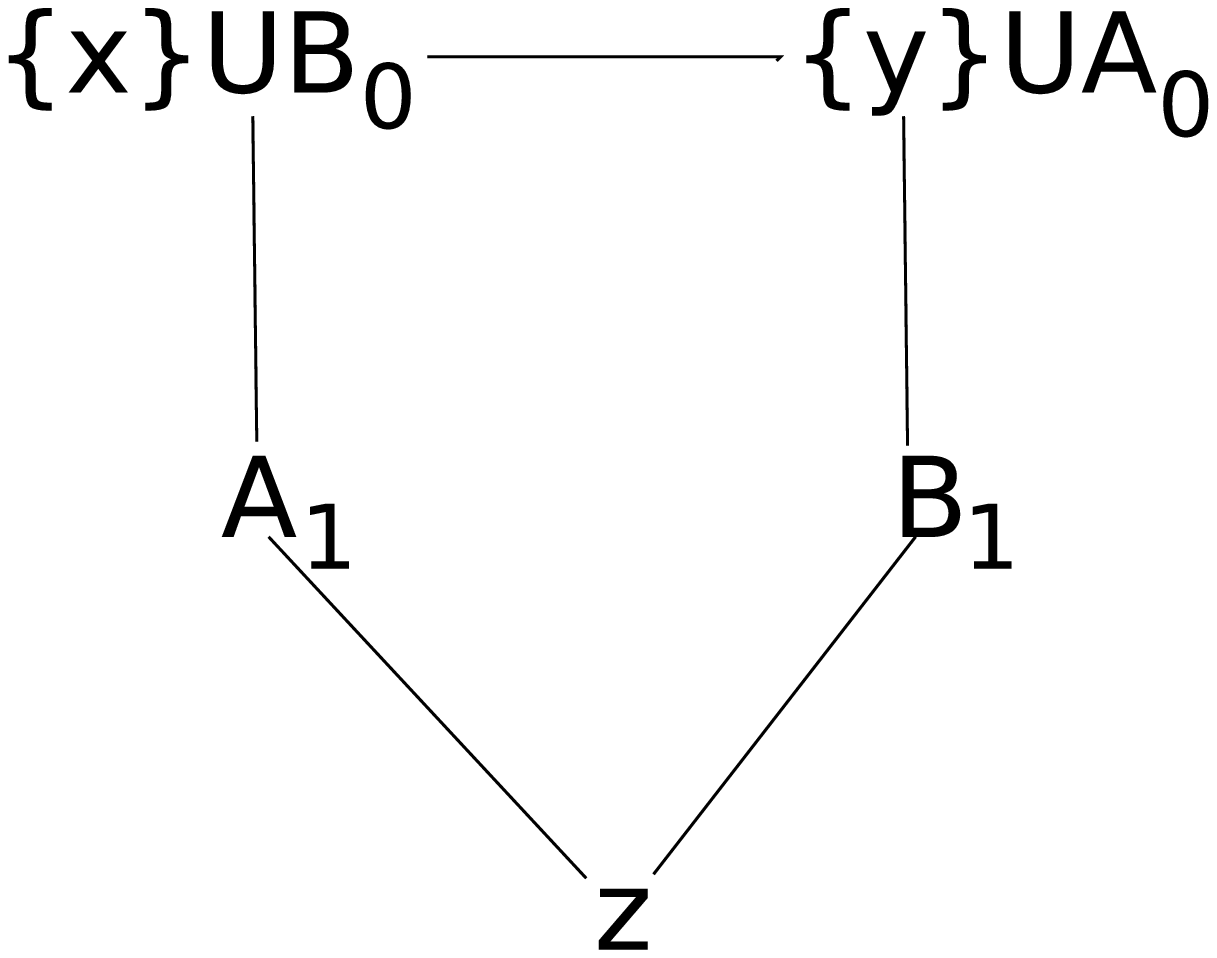}
\end{tabular}
\end{center}
where an edge is a complete bipartite graph on sets incident to it,
and there are no edges inside these sets.

If $n\ge 6$, then there are at least two elements in one of the sets
$\{x\}\cup B_0$, $\{y\}\cup A_0$, $A_1$ or $B_1$. However, these two
vertices have the same neighbourhood and hence they cannot be
separated by the other $n-2\ge 4$ vertices. Thus, $n=5$. 

If
$n=5$, and both $A_1$ and $B_1$ are non-empty, we must have
$A_0=B_0=\emptyset$ and $G=C_5$, which is in $\Gr(5,4)$ by
\refE{exa}\ref{exaC}. 

Finally, assume $n=5$ and $A_1=\emptyset$ (the case $B_1=\emptyset$
is the same after relabelling). Then $B_1$ is non-empty, since $G$ is
connected. If $B_0$ is non-empty, let $b_0\in B_0$ and  $b_1\in B_1$,
and observe that \set{x,b_0,b_1,z} does not separate $z$ and $b_1$.
Hence $B_0=\emptyset$. We thus have either $|A_0|=1$ and $|B_1|=1$, or
$|A_1|=0$ and $|B_1|=2$, and both cases yield
the graph (d) in Figure~\ref{Figgr54}, which easily is seen to be in
$\Gr(5,4)$. 

(ii) Suppose that there is no induced subgraph $K_1\cupx K_2$.
Since $G$ is connected, we can find an edge $x\sim y$. 
Let, as above, 
$A=N(x)\setminus \{y\}$ and $B=N(y)\setminus \{x\}$.
If $a\in A$ and $b\in B$ and $a\not\sim
b$, then $(\{a,x,b\},\{\{a,x\}\})$ is an induced subgraph and we are back
in case (i). Hence, all edges between
sets $A$ and $B$ exist and thus, recalling that $G$ has no triangles,
$G$ is the complete bipartite graph
with bipartition $(A\cup\set y,B\cup\set x)$.
By
Example~\ref{exa}\ref{exaKK}, then $n\le 5$. If $n=5$, we get $G=K_{2,3}$
or $G=K_{1,4}=S_4$, which both belong to $\Gr(5,4)$ by \refE{exa}\ref{exaKK}.
\end{ex}

We summarize the result in a theorem.

\begin{thm}\label{exgr54}
$\NN(4)=5$. 
More precisely, $\Gr(4)=\Gr(4,4)\cup\Gr(5,4)$, where
$\Gr(4,4)=\set{C_4,P_4,S_4,E_4,P_3\cupx K_1}$ and
$\Gr(5,4)$ consists of the four graphs in
 Figure~\ref{Figgr54}. 
\end{thm}

\begin{figure}\caption{All the different graphs in $\Gr(5,4)$.}\label{Figgr54}
\begin{center}
\begin{tabular}{cccc}
a) \includegraphics[width=2cm]{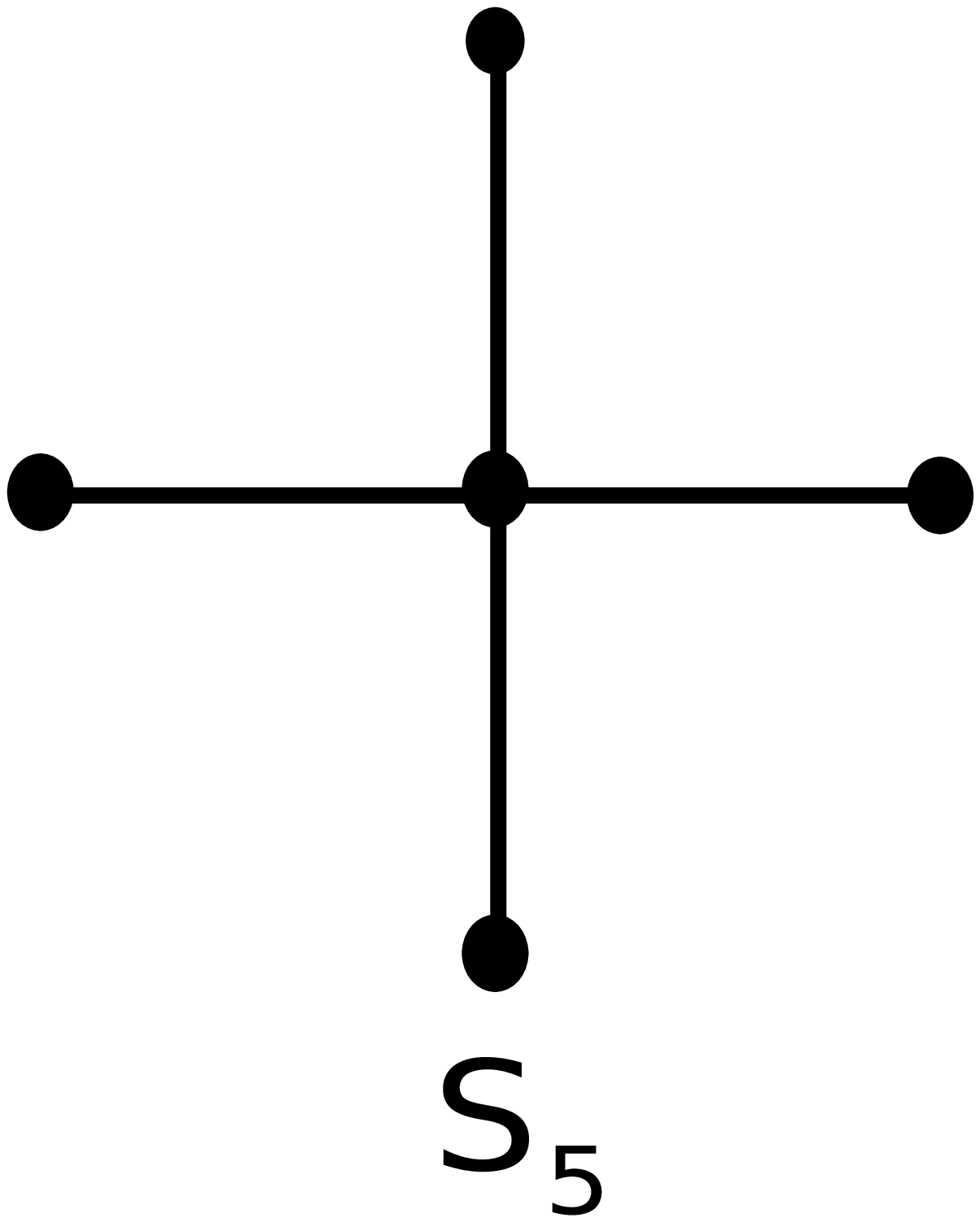}
&
b) \includegraphics[width=2cm]{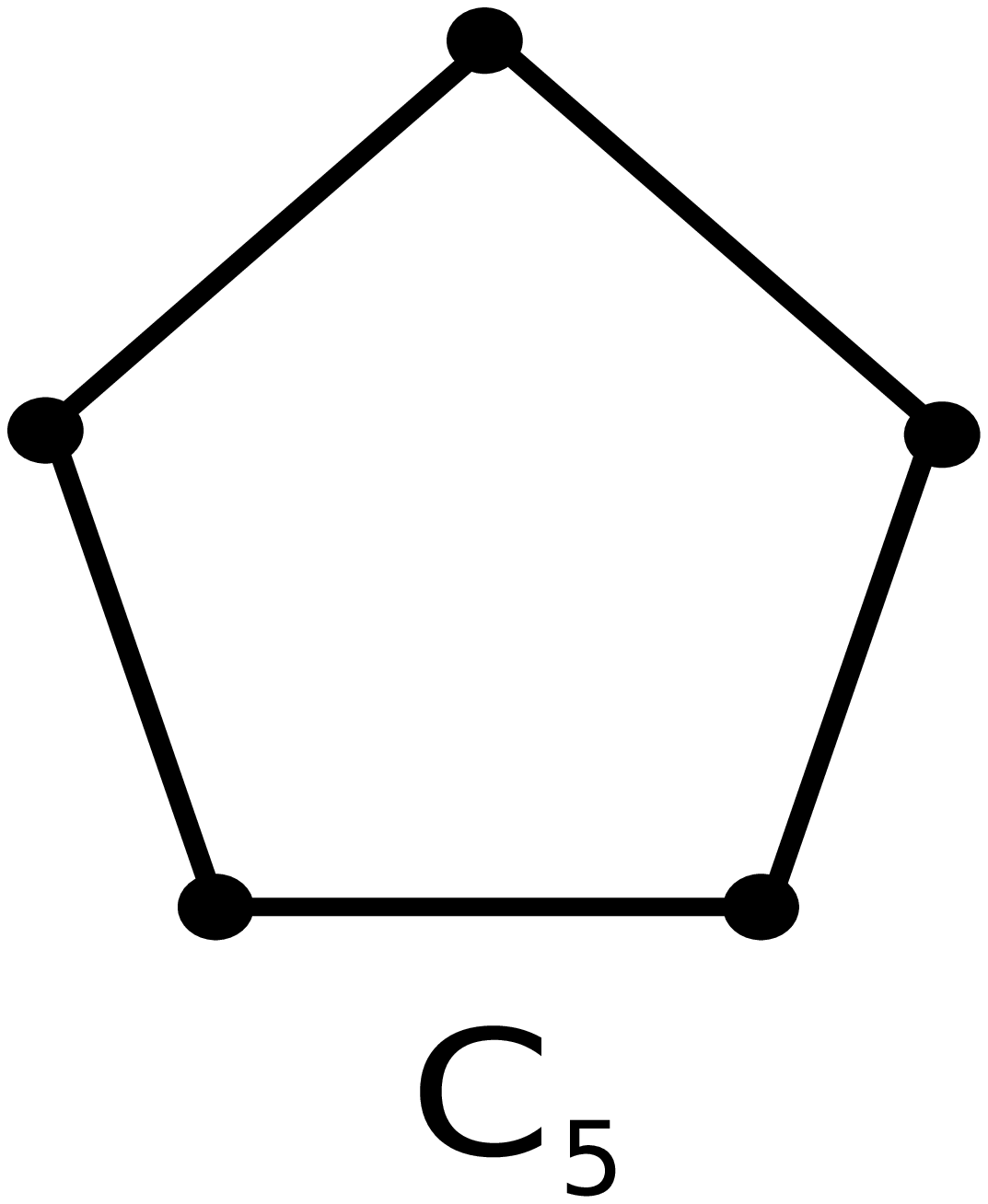}
&
c) \includegraphics[height=2.5cm]{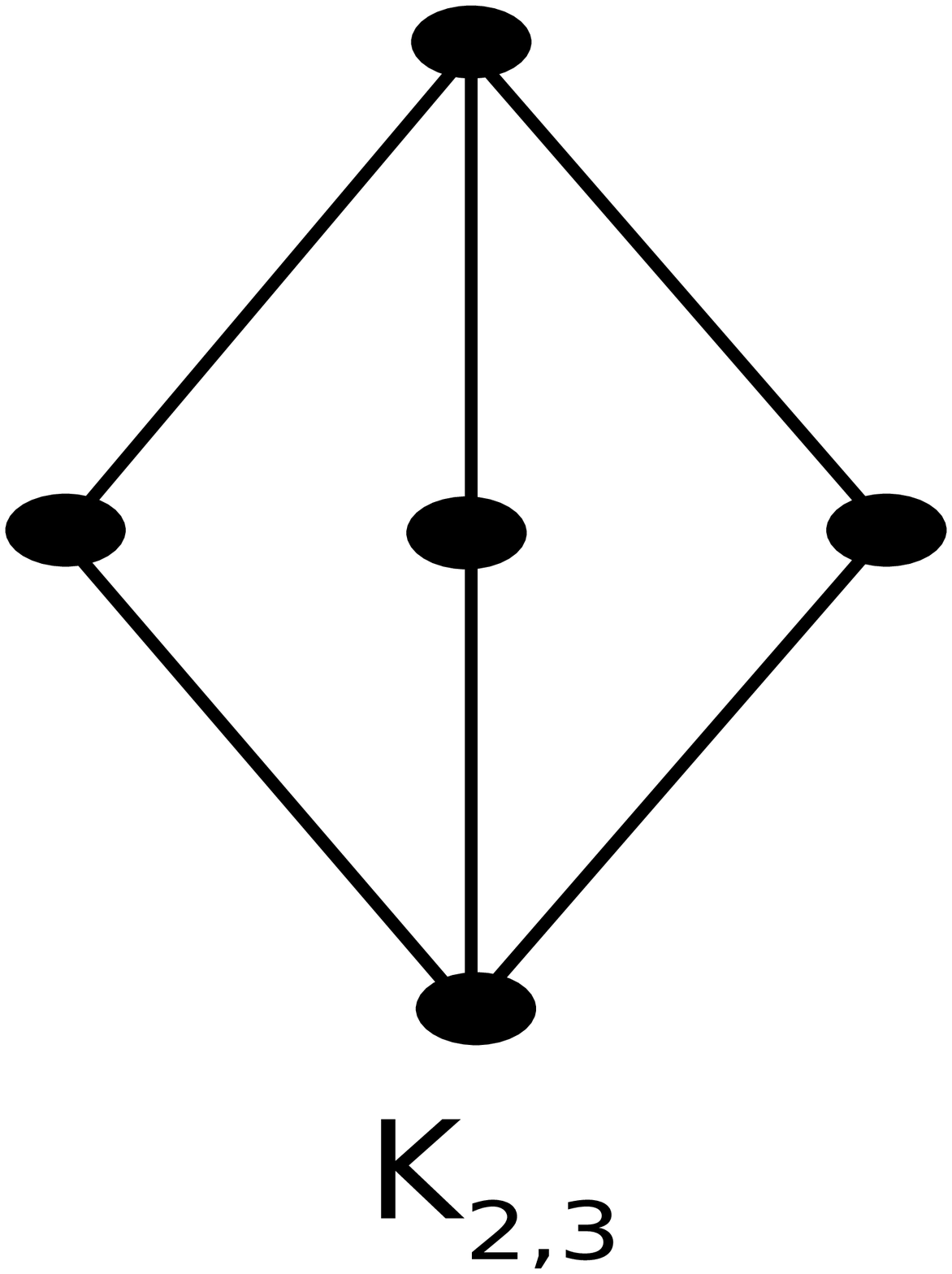}
&
d) \includegraphics[height=2.5cm]{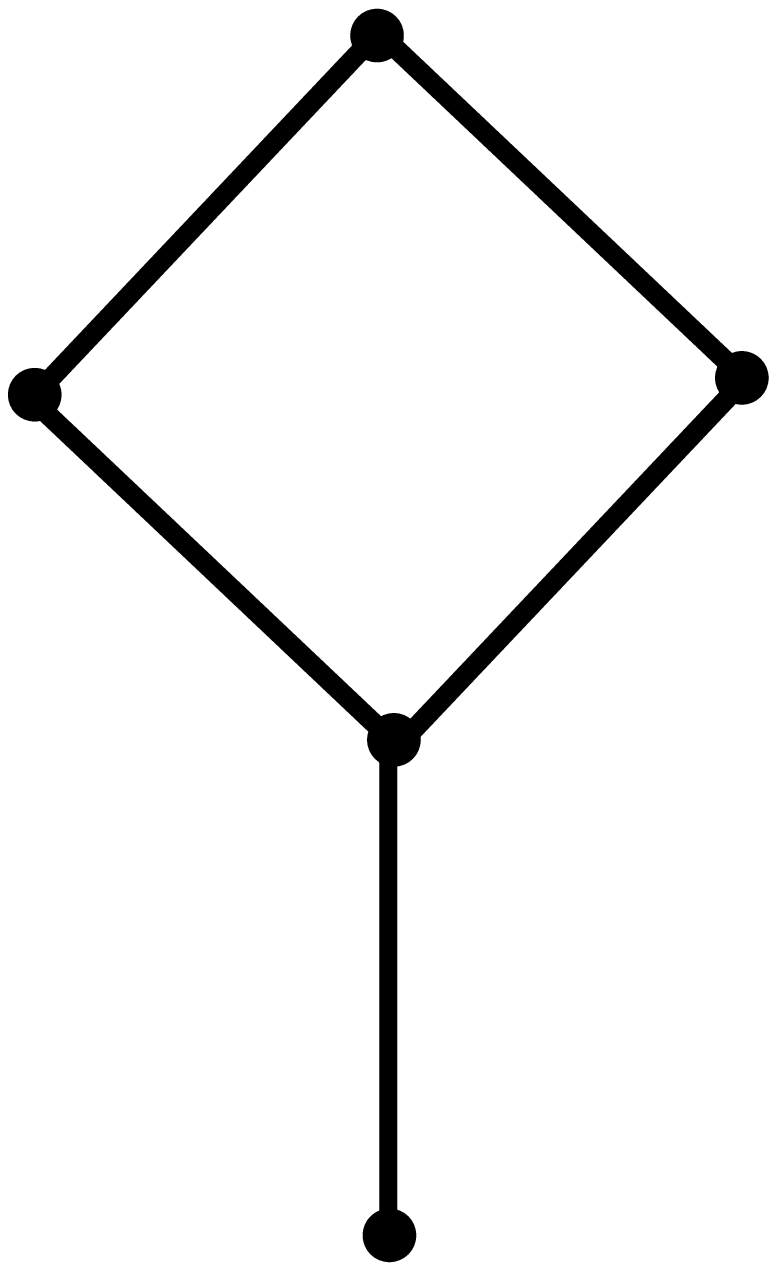}
\end{tabular}
\end{center}
\end{figure}

For $k=5$ and 6, we do not describe $\Gr(k)$ completely, but we find
$\NN(k)$, using some results that will be proved in \refS{Supper}.
Upper and lower bounds for some other values of  $k$ are given in
Table~\ref{tab}.

\begin{thm}\label{T5}\label{T6}
  $\NN(5)=8$, $\NN(6)=9$ and $11\le \NN(7)\le 12$.
\end{thm}

\begin{proof}
First observe that $\NN(5)\ge 8$ since the 3-dimensional cube
belongs to $\Gr(8,5)$ by \refE{Ecube}. The upper bound follows from Theorem~\ref{kbound}.

\refE{Ebcc} gives an example (a centred cube) showing that
$\NN(6)\ge9$.
(Another example is given by the Paley graph $P(9)$, see
\refT{paley}.)
The upper bound is given by \refT{upperbound} in \refS{Supper}.

The construction of a graph in $\Gr(11,7)$ is given in
Figure~\ref{figGraph}. The upper bound follows both from
Theorem~\ref{upperbound} and Theorem~\ref{kbound}. 
%
%
%
\end{proof}

%


\begin{center}
\begin{figure}   
\caption{A graph in $\Gr(11,7)$ found by a computer search.}
\label{figGraph}
\begin{center}
\includegraphics[width=4 cm]{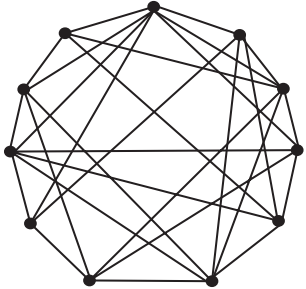}
\end{center}
\end{figure}
\end{center}

\section{Upper estimates on the order}\label{Supper}

In the next theorem we give an upper on bound on $\NN(k)$, which is
obtained using knowledge on error-correcting codes. 

\begin{thm}\label{kbound}
If $k\ge 2$, then
$\NN(k)\le 2k-2$.
\end{thm}

\begin{proof}
We begin by giving a construction from a graph in $\Gr(n,k)$ to
error-correcting codes. A non-existence result of error-correcting
codes then yields the non-existence of $\Gr(n,k)$ graphs of certain
parameters. Let $G=(V,E)\in \Gr(n,k),$ where $V=\{x_1,x_2,\ldots,x_n\}$. 
We construct $n+1$ binary strings $\mathbf{y}_i=(y_{i1},\ldots,y_{in})$ 
of length $n$, for $i=0,\dots,n$,  
from the sets $\emptyset=N[\emptyset]$ and $N[x_i]$ for $i=1,\ldots, n$
by defining $y_{0j}=0$ for all $j$ and
\[y_{ij}=
\begin{cases}
  0 &\text{if }x_j\not\in N[x_i]\\
1&\textnormal{if }x_j\in N[x_i]
\end{cases}
,
\qquad 
1\le i\le n.
\] 
Let $C$ denote the code which consists of these binary strings as codewords.
Because $G\in \Gr(n,k)$, the symmetric difference of two closed
neighbourhoods $N[x_i]$ and $N[x_j]$, or of one neigbourhood $N[x_i]$
and $\emptyset$, 
is at least $n-k+1$ by
\eqref{eqkarak};  in other words, the minimum Hamming distance $d(C)$
of the code $C$ is at least $n-k+1$. 

We first give a simple proof that $\NN(k)\le 2k-1$. Thus,
suppose that there is a $G\in \Gr(n,k)$ such that $n=2k$. In
the corresponding error-correcting code $C$, the minimum distance is
at least $d=n-k+1=k+1>n/2$. 
Let the maximum cardinality of the
error-correcting codes of length $n$ and minimum distance at least $d$
be denoted by $A(n,d)$. 
We can apply the Plotkin bound (see for example
\cite[Chapter 2, \S 2]{macwil}), which says $A(n,d)\le 2\lfloor
d/(2d-n)\rfloor,$ when $2d>n$. Thus, we have 
\[A(n,d)\le 
2\left\lfloor\frac{k+1}{2}\right\rfloor\le k+1.\] 
Because $k+1<2k= n<|C|$, this contradicts the existence of $C$. Hence,
there cannot exist a graph $G\in \Gr(2k,k)$, and thus
$\Gr(n,k)=\emptyset$ when $n\ge 2k$. 

The Plotkin bound is not strong enough to imply $\NN(k)\le 2k-2$ in
general, but we obtain this from the proof of the Plotkin bound as follows.
(In fact, for odd $k$, $\NN(k)\le 2k-2$ follows from the Plotkin bound
for an odd minimum distance. We leave this to the reader since the
argument below is more general.)

Suppose that $G=(V,E)\in\Gr(n,k)$ with $n=2k-1$. 
We thus have a corresponding error-correcting code $C$ with
$|C|=n+1=2k$ and minimum Hamming distance at least $n-k+1=k$. Hence,
letting $d$ denote the Hamming distance,
\begin{equation}
  \label{pl}
\sum_{0\le i<j\le n} d(y_i,y_j) \ge \binom{n+1}2k =
\frac{2k(2k-1)}{2}k
=(2k-1)k^2.
\end{equation}
On the other hand, if there are $s_m$ strings $y_i$ with $y_{im}=1$,
and thus $|C|-s_m=2k-s_m$ strings with $y_{im}=0$, then the number of
ordered pairs $(i,j)$ such that $y_{im}\neq y_{jm}$ is
$2s_m(2k-s_m)\le 2k^2$. Hence each bit contributes at most $k^2$ to the
sum in \eqref{pl}, and summing over $m$ we find
\begin{equation}
  \label{pl2}
\sum_{0\le i<j\le n} d(y_i,y_j) 
\le nk^2
=(2k-1)k^2.
\end{equation}
Consequently, we have equality in \eqref{pl} and \eqref{pl2}, and thus
$d(y_i,y_j)=k$ for all pairs $(i,j)$ with $i\neq j$.

In particular, $|N[x_i]|=d(y_i,y_0)=k$ for $i=1,\dots,n$, and thus
every vertex in $G$ has degree $k-1$, i.e., $G$ is $(k-1)$-regular.
Hence, $2|E|=n(k-1)=(2k-1)(k-1)$, and $k$ must be odd.

Further, if $i\neq j$, then $|\Nxi\symdiff\Nxj|=d(y_i,y_j)=k$, and
since $\Nxi\setminus\Nxj$ and $\Nxj\setminus\Nxi$ have the same size
$k-|\Nxi\cap\Nxj|$, they have both the size $k/2$ and $k$ must be
even.

This contradiction shows that $\Gr(2k-1,k)=\emptyset$, and thus
$\NN(k)\le2k-2$. 
\end{proof}

The next theorem (which does not use \refT{kbound}) 
will lead to another
upper bound in \refT{upperbound}.
It can be seen as an improvement for
the extreme case $\Gr(2k-2,k)$ of Mantel's \cite{Mantel}
theorem on existence of triangles in a graph. 
Note that this result fails for $k=5$ by \refE{Ecube}.
\begin{thm}\label{trianglelemma}
Suppose $G\in \Gr(n,k)$ and $k\ge 6$. If $n\ge 2k-2$, then there is a
triangle in $G$.
\end{thm}
\begin{proof}
Let $G=(V,E)\in \Gr(n,k)$. Suppose to the contrary that there are no
triangles in $G$. If there is a vertex $x\in V$ such that
$\deg(x)\ge k+1$, then we select in $N(x)$ a $k$-set $X$ and a vertex $y$
outside it; since $X$ has to dominate $y$,
it is clear that there exists a triangle $xyz$. 
Hence $\deg(x)\le k$ for every $x$.
On
the other hand, we know that for all $x\in V$ $\deg(x)\ge n-k\ge
k-2$.

Let $x\in V$ be a vertex whose degree is minimal. We denote
$V\setminus N[x]=B$ and we use the fact that $|B|\le k-1$.

1) Suppose first $\deg(x)=k$. Because $\deg(x)$ is minimal we know
that for all $a\in N(x)$, $\deg(a)=k$. This is possible if and only
if $|B|=k-1$ and for all $a\in N(x)$ we have $B\cap N(a)=B$. But
then in the $k$-subset $C=\{x\}\cup B$ we have $I(C;a)=I(C;b)$ for all
$a,b\in N(x)$. This is impossible.

2) Suppose then $\deg(x)=k-1$. If now $|B|\le k-2$ the graph is
impossible as in the first case. Hence, $|B|=k-1$. For every $a\in
N(x)$ there are at least $k-2$ adjacent vertices in $B$, and thus at
most 1 non-adjacent. This
implies that for all $a, b\in N(x)$, $a\ne b$, we have $|N(a)\cap
N(b)\cap B|\ge k-3\ge 2$, when $k\ge 5$. Hence, by choosing $a,b\in
N(x)$, $a\ne b$, we have the $k$-subset $C=\{x\}\cup (N(x)\setminus
\{a,b\})\cup \{c_1,c_2\}$, where $c_1,c_2\in N(a)\cap N(b)\cap B$.
In this $k$-subset $I(C;a)=I(C;b)$, which is impossible.

3) Suppose finally $\deg(x)=k-2$. Now $|B|=k-1$, otherwise we cannot
have $n\ge 2k-2$. If there is $b\in B$ such that $|N(b)\cap
N(x)|=k-2$, then because $\deg(b)\le k$ we have
$|B\setminus(N[b]\cap B)|\ge k-4\ge 2$, when $k\ge 6$. Hence, there
are $c_1,c_2\in B\setminus N[b]$, $c_1\ne c_2$, and in the $k$-subset
$C=N(x)\cup \{c_1,c_2\}$ we have $I(C;x)=I(C;b)$ which is impossible.

Thus, for all $b\in B$ we have $|N(b)\cap N(x)|\le k-3$. On the
other hand, 
each of the $k-2$ vertices in $N(x)$ has at least $k-3$ adjacent
vertices in $B$, so
the vertices in $B$ have on the average at least $(k-2)(k-3)/(k-1)>k-4$
adjacent vertices in the set $N(x)$. Hence, we can find $b\in B$ such 
that $|N(b)\cap N(x)|=k-3$. Because $\deg(b)\ge k-2$ we have at
least one $b_0\in B$ such that $d(b,b_0)=1$. Because there are no
triangles, each of the $k-3$ neighbours of $b$ in $N(x)$ is not
adjacent with $b_0$, and therefore adjacent to at least $k-3$ of the
$k-2$ vertices in $B\setminus\set{b_0}$. 
Hence, for all $a_1,a_2\in N(x)\cap N(b)$,
$a_1\ne a_2$, we have $|N(a_1)\cap N(a_2)\cap B|\ge k-4\ge 2$ when
$k\ge 6$. In the $k$-subset $C=\{x,b_0,c_1,c_2\}\cup (N(x)\setminus
\{a_1,a_2\})$, where $c_1,c_2\in N(a_1)\cap N(a_2)\cap B$, we have
$I(C;a_1)=I(C;a_2)$, which is impossible.
\end{proof}

\begin{lemma}\label{Ltriangleupperbound}
If there is a graph $G\in\Gr(n,k)$ that contains a triangle, then
$n\le 3k-9$. (In particular, $k\ge5$.)
\end{lemma}
\begin{proof}
Suppose that $G=(V,E)\in\Gr(n,k)$ and that
there is a triangle $\{x,y,z\}$ in $G$. 
Let, for $v,w\in V$, $J_w(v)$ denote the indicator function given by
$J_w(v)=1$ if $v\in N[w]$ and $J_w(v)=0$ if $v\notin N[w]$.
Define the set
$M_{xy}=\set{v\in V:J_x(v)=J_y(v)}$, and
$M'_{xy}=M_{xy}\setminus\set{x,y,z}$.
Since $M_{xy}$ does not
separate $x$ and $y$, we have $|M_{xy}|\le k-1$. Further,
$\set{x,y,z}\subseteq M_{xy}$, and thus $|M'_{xy}|\le k-4$.
Define similarly $M_{xz}$, $M_{yz}$, $M'_{xz}$, $M'_{yz}$; the same
conclusion holds for these.

Since the indicator functions take only two values,
$M_{xy}$, $M_{xz}$ and $M_{yz}$ cover $V$, and thus
\begin{equation*}
  n=|V|= |M'_{xy}\cup M'_{xz}\cup M'_{yz}\cup\set{x,y,z}|
\le 3(k-4)+3=3k-9.
\end{equation*}
Since $n\ge k$, this entails $3k-9\ge k$ and thus $k\ge 5$.
\end{proof}

The following upper bound is generally weaker than
Theorem~\ref{kbound}, but it gives the optimal result for $k=6$. 
(Note that the result fails for $k\ge5$, see \refS{Ssmall}.)
\begin{thm}\label{upperbound}
Suppose $k\ge 6$. Then $\NN(k)\le 3k-9$.
\end{thm}
\begin{proof}
Suppose that $G\in\Gr(n,k)$.
If $G$ does not contain any triangle, then
\refT{trianglelemma} yields $n\le 2k-3 \le 3k-9$.
If $G$ does contain a triangle, then
Lemma~\ref{Ltriangleupperbound} yields $n\le 3k-9$.
\end{proof}

\section{Strongly regular graphs}\label{sec:strong}
A graph $G=(V,E)$ is called \emph{strongly regular} with
parameters $(n,t,\lambda, \mu)$ if $|V|=n$, $\deg(x)=t$ for all $x\in
V$, any two adjacent vertices have exactly $\lambda$ common
neighbours, and any two nonadjacent vertices have exactly $\mu$ common
neighbours;
we then say that $G$ is a $(n,t,\lambda,\mu)\SRG$. 
See \cite{DistRegGraphs} for more information. 
By \cite[Proposition 1.4.1]{DistRegGraphs} we know that if 
$G$ is a $(n,t,\lambda,\mu)\SRG$, then $n= t+1+t(t-1-\lambda)/\mu$. 

We give two examples of strongly regular graphs that will be used below.

\begin{ex} \label{Epaley}
The well-known Paley graph
$P(q)$, where $q$ is a prime power with $q\equiv1\pmod4$,
is a 
$(q,\,(q-1)/2,\,(q-5)/4,\,(q-1)/4)\SRG$, 
see for example \cite{DistRegGraphs}.  
The vertices of $P(q)$ are the elements of the finite field $F_q$,
with an edge $ij$ if and only if $i-j$ is a non-zero square in the
field; when $q$ is a prime, this means that the vertices are
\set{1,\dots,q} with edges $ij$ when $i-j$ is a quadratic residue
mod $q$.
\end{ex}

\begin{ex}\label{ERSHCD}
Another construction of strongly regular graphs uses a regular
symmetric Hadamard matrix with constant diagonal (\RSHCD)
\cite{GothSei70},
\cite{BvL84},
\cite{CombDesI}.
In particular, in the case (denoted \RSHCD+) of 
a regular symmetric $n\times n$ Hadamard matrix $H=(h_{ij})$ with
diagonal entries $+1$ and constant positive row sums $2m$ (necessarily
even when $n>1$), then $n=(2m)^2=4m^2$ and the graph $G$ with vertex
set \set{1,\dots,n} and an edge $ij$ (for $i\neq j$) if and only if
$h_{ij}=+1$ is a $(4m^2,\,2m^2+m-1,\,m^2+m-2,\,m^2+m)\SRG$
\cite[\S8D]{BvL84}.

It is not known for which $m$ such \RSHCD+ exist
(it has been conjectured that any $m\ge1$ is possible)
but constructions for many $m$ are known, see
\cite{GothSei70},
\cite[V.3]{Wallis} and \cite[IV.24.2]{CombDesI}. 
For example, starting with the $4\times4$ \RSHCD+ 
\begin{equation*}
H_4=
  \begin{pmatrix}
	\phantom{-}1 & \phantom{-}1 & \phantom{-}1 & -1 \\
	\phantom{-}1 & \phantom{-}1 & -1 & \phantom{-}1 \\
	\phantom{-}1 & -1 & \phantom{-}1 & \phantom{-}1 \\
	-1 & \phantom{-}1 & \phantom{-}1 & \phantom{-}1 
  \end{pmatrix}
\end{equation*}
its tensor power $H_4^{\otimes r}$ is an \RSHCD+ with $n=4^r$, and thus
$m=2^{r-1}$, for any $r\ge1$. This yields a 
$(2^{2r},\,2^{2r-1}+2^{r-1}-1,\,2^{2r-2}+2^{r-1}-2,\,2^{2r-2}+2^{r-1})\SRG$
with vertex set
$\set{1,2,3,4}^r$, where two different vertices $(i_1,\dots,i_r)$ and
$(j_1,\dots,j_r)$ are adjacent if and only if the number of
coordinates $\nu$ such that $i_\nu+j_\nu=5$ is even.
\end{ex}

\begin{thm}\label{TSRG}
A strongly regular graph $G=(V,E)$ with parameters $(n,t,\lambda,\mu)$
belongs to $\Gr(n,k)$ if and only  if 
\begin{equation*}
  k\ge \max\bigl\{n-t,\, n-2t+2\gl+3,\, n-2t+2\mu-1 
\bigr\},
\end{equation*}
or, equivalently,
$t\ge n-k$ and $2\max\{\lambda+1,\mu-1\}\le k+2t-n-1$.
\end{thm}

\begin{proof}
An immediate consequence of Theorem~\ref{karakterisointi}, since
$|\Nx|=t+1$ for every vertex $x$ and
$|N[x]\symdiff N[y]|$ equals 
$2(t-\gl-1)$ when $x\sim y$ and $2(t+1-\mu)$ when
$x\not\sim y$, $x\neq y$. 
\end{proof}

We can extend this construction to other values of $n$ by modifying
the strongly regular graph.

\begin{thm}\label{TSRG2}
  If there exists a strongly regular graph with parameters
  $(n_0,t,\gl,\mu)$, then for every 
$i=0,\dots,n_0+1$ there exists a graph in $\Gr(n_0+i,k_0+i)$,
where
\begin{equation*}
 k_0= \max\bigl\{n_0-t,\,t,\, n_0-2t+2\gl+3,\, n_0-2t+2\mu-1,\,
 2t-2\gl-1,\,2t-2\mu+2 \bigr\},
\end{equation*}
provided $k_0\le n_0$.
\end{thm}

\begin{proof}
For $i=0$, this is a weaker form of \refT{TSRG}.
For $i\ge1$, we
suppose that $G_0=(V_0,E_0)$ is $(n_0,t,\gl,\mu)\SRG$
and build a graph $G_i$ in $\Gr(n_0+i,k_0+i)$ from $G_0$ by adding suitable new
vertices and edges. 

If $1\le i\le n_0$, choose $i$ different
vertices $x_1,x_2,\dots,x_i$ in $V_0$. Construct a new
graph $G_i=(V_i,E_i)$ by taking $G_0$ and adding to it new vertices
$x'_1,x'_2,\dots,x'_i$ and new edges $x_j'y$ for $j\le i$ and all
$y\notin N_{G_0}(x_j)$.

First, $\deg_{G_i}(x)\geq \deg_{G_0}(x)=t$  for
$x\in V_0$ and $\deg_{G_i}(x')=n_0-t$ for $x'\in V_i'=V_i\setminus V_0$.
We proceed to investigate $\Nx \symdiff N[y]$, and separate several cases.

(i) If $x,y\in V_0$, with $x\neq y$, then
\begin{equation*}
  \begin{split}
  \bigabs{\Nx \symdiff N[y]}
&\ge \bigabs{(\Nx \symdiff N[y])\cap V_0}
=
 \bigabs{(N_{G_0}[x] \symdiff N_{G_0}[y])},
  \end{split}
\end{equation*}
which equals
$2(t-\gl-1)$ if $x\sim y$ and $2(t-\mu+1)$ if $x\not\sim y$.

(ii) If $x\in V_0$, $y'\in V_i'$, then, since $\symdiff$ is
associative and commutative,
\begin{equation*}
\bigabs{(\Nx \symdiff N[y'])\cap V_0}
=
 \bigabs{(N_{G_0}[x] \symdiff (V_0\symdiff N_{G_0}(y))}
=n_0 - \bigabs{(N_{G_0}[x] \symdiff N_{G_0}(y))},
\end{equation*}
which equals $n_0-1$ if $x=y$, $n_0-(2t-2\gl-1)$ if $x\sim y$, and
$n_0-(2t-2\mu+1)$ if $x\not\sim y$ and $x\neq y$.
If $x\sim y$, further,
$\bigabs{(\Nx \symdiff N[y'])\cap V_i'}\ge1$, since $y'\not\in \Nx$.

(iii)
If $x',y'\in V_i'$, with $x'\neq y'$, then
\begin{equation*}
\bigabs{(N[x'] \symdiff N[y'])\cap V_0}
=
 \bigabs{(V_0\setminus N_{G_0}(x)) \symdiff (V_0\setminus N_{G_0}(y))}
=
 \bigabs{(N_{G_0}(x) \symdiff N_{G_0}(y))},
\end{equation*}
which equals $2(t-\gl)$ if $x\sim y$ and $2(t-\mu)$ if $x\not\sim
y$. Further, $\bigabs{(N[x'] \symdiff N[y'])\cap V_i'}=|\set{x',y'}|=2$.

Collecting these estimates, we see that $G_i\in\Gr(n_0+i,k_0+i)$ by
Theorem~\ref{karakterisointi} (or \refC{Cekaehto}) with our choice of
$k_0$. Note that $2k_0\ge (n_0-2t+2\gl+3)+(2t-2\gl-1)=n_0+2\ge3$, so $k_0\ge2$.

Finally, for $i=n_0+1$, we construct $G_{n_0+1}$ by
adding a new vertex to $G_{n_0}$ and connecting
it to all other vertices.
The graph $G_{n_0}$ has by construction maximum degree
$\Delta_{G_{n_0}}=n_0\le k_0+n_0-2$.
Hence, \refL{Ladd1} shows that $G_{n_0+1}\in\Gr(n_0+1,k_0+n_0+1)$.
\end{proof}

We specialize to the Paley graphs, and obtain from \refE{Epaley} and
Theorems \ref{TSRG}--\ref{TSRG2} the following.

\begin{thm} \label{paley}
Let $q$ be an odd prime power such that $q\equiv 1 \pmod 4$.
\begin{romenumerate}
  \item
The Paley graph $P(q)\in\Gr(q,(q+3)/2)$. 
\item
There exists a graph in $\Gr(q+i,(q+3)/2+i)$ for all
$i=0,1,\dots,q+1$.
\end{romenumerate}
\end{thm}

Note that the rate $2q/(q+3)$ for the Paley graphs approaches 2 as
$q\to\infty$; in fact, with $n=q$ and $k=(q+3)/2$ we have $n=2k-3$,
almost attaining the bound $2k-2$ in \refT{kbound}.
(The Paley graphs thus almost attain the bound in
\refT{kbound}, but never attain it exactly.)

\begin{cor}
  \label{Clower}
$\NN(k)\ge 2k-o(k)$ as $k\to\infty$.
\end{cor}

\begin{proof}
 Let $q=p^2$ where (for $k\ge6$) $p$ is the largest prime such that
 $p\le \sqrt{2k-3}$. It follows from the prime number theorem that 
 $p/\sqrt{2k-3}\to1$ as $k\to\infty$, and thus $q=2k-o(k)$.
Hence, if $k$ is large enough, then $k\le q\le 2k-3$, and \refT{paley}
 shows that  
 $P(q)\in\Gr(q,(q+3)/2)\subseteq\Gr(q,k)$, so $\NN(k)\ge q=2k-o(k)$.
(Alternatively, we may let $q$ be the largest prime such that $q\le
 2k-3$ and $q\equiv  1\pmod 4$ and use the prime number theorem for
 arithmetic progressions  \cite[Chapter 17]{Huxley} to see that then
 $q=2k-o(k)$.) 
\end{proof}

We turn to the strongly regular graphs constructed in
\refE{ERSHCD} and find from \refT{TSRG} that they are in 
$\Gr(4m^2,2m^2+1)$, thus attaining the bound in \refT{kbound}.
We state that as a theorem.

\begin{thm}\label{TRSHCD}
  The strongly regular graph constructed in \refE{ERSHCD} from an
  $n\times n$ \RSHCD+  belongs to $\Gr(n,n/2+1)$. 
\end{thm}

\begin{cor}\label{Cinfty}
  There exist infinitely many integers $k$ such that $\NN(k)=2k-2$.
\end{cor}

\begin{proof}
  If $k=n/2+1$ for an even $n$ such that there exists an 
  $n\times n$ \RSHCD+, then $\NN(k)\ge n=2k-2$ by \refT{TRSHCD}.
The opposite inequality is given by \refT{kbound}. By \refE{ERSHCD},
  this holds at least for $k=2^{2r-1}+1$ for any $r\ge1$.
\end{proof}

\section{Smaller identifying sets}\label{Ssmaller}

The fact that \emph{all} sets of $k$ vertices in a given graph
are identifying implies typically
that there exist \emph{many} identifying sets of smaller size $s$ too,
as is shown by the following result.

\begin{thm}\label{prob} Let $G=(V,E)\in \Gr(n,k)$. Then,
for a random subset $S$ of $V$ of size $s$
\[
\P(\text{$S$ is identifying in $G$}) \geq
1-\binom{n+1}{2}\frac{\binom{k-1}{s}}{\binom{n}{s}}.
\]
\end{thm}

\begin{proof}
Let $G=(V,E)\in \Gr(n,k)$ and $\mathcal{S}$ be the set of all
$s$-subsets of $V$. Clearly, $|\mathcal{S}|=\binom{n}{s}$. Denote by
$F_2(S)$, $S\in \mathcal{S}$, the number of unordered pairs
$\{u,v\}\in \binom{V}{2}$ such that $u$ and $v$ are not separated by $S$,
that is, $I(S;u)\symdiff I(S;v)=\emptyset$, and by $F_1(S)$
the number of vertices $w\in V$ such that $I(S;w)=\emptyset$.

We count
\begin{align*}
\sum_{S\in \mathcal{S}} F_2(S)+\sum_{S\in \mathcal{S}} F_1(S) & = 
\sum_{S\in \mathcal{S}}\sum_{\quad \substack{\{u,v\}\in \binom{V}{2}\\
I(u) \symdiff I(v)=\emptyset}} 1
+\sum_{S\in \mathcal{S}}
 \sum_{\quad \substack{w\in V\\ I(w)=\emptyset}} 1\\
  &=  \sum_{\{u,v\}\in \binom{V}{2}} \sum_{\quad \substack{S\in \mathcal{S}\\
I(u) \symdiff I(v)=\emptyset}}1 
+ \sum_{w\in V} \sum_{\quad \substack{S\in \mathcal{S}\\I(w)=\emptyset}}1\\
& \leq  \lrpar{\binom{n}{2}+n}\binom{k-1}{s}
=\binom{n+1}{2}\binom{k-1}{s}.
\end{align*}
This bounds from above the number of sets $S\in \mathcal{S}$ that
have an unidentified pair or a vertex with empty $I$-set. Thus
\[\P(S\in \mathcal{S}  \mbox{ is identifying})\geq
1-\frac{\binom{n+1}{2}\binom{k-1}{s}}{\binom{n}{s}} .
\qedhere
\]
\end{proof}

It follows that for many graphs, for example Paley graphs, 
almost all $s$-subsets are
identifying even when $s$ is not too far away from the smallest value where
there exists any identifying subset.
We illustrate this for $P(29)$ in Figure~\ref{probpic},
and state the following consequences. 

\begin{center}
\begin{figure}
\caption{The bound in Theorem~\ref{prob} for the graphs in
$\Gr(29,16)$.} \label{probpic}
\begin{center}
\epsfig{file=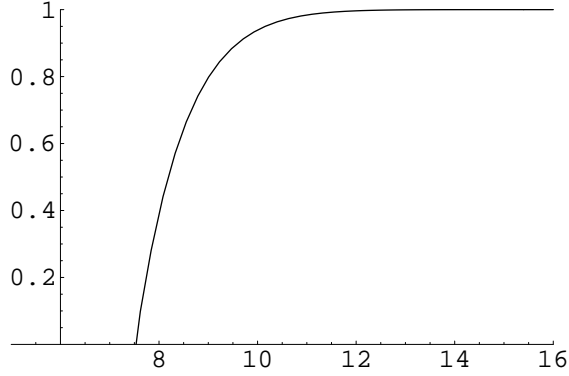}
\end{center}
\end{figure}
\end{center}

\begin{thm}\label{Tid}
  If $G\in\Gr(n,k)$ with $k\ge2$ and 
$s$ is an integer with 
\[\log\binom{n+1}2/\log(n/(k-1))<s\le n,\]
then there exists an identifying  $s$-set of vertices of $G$.
\end{thm}

\begin{proof}
  If $s\ge k$, then every $s$-set will do, so suppose $s\le k-1$.
Then 
\begin{equation*}
\frac{ \binom{k-1}s}{\binom ns} \le \lrpar{\frac{k-1}n}^s 
<e^{-\log\binom{n+1}2},
\end{equation*}
and \refT{prob} shows that there is a positive probability that a
random $s$-set is identifying.
\end{proof}

\begin{thm}
For the Paley graphs,
\[\min\{|S|: \mbox{$S$ is identifying in  $P(q)$}\}=
\Theta(\log{q}).
\]
\end{thm}

\begin{proof}
Theorems \ref{paley} and \ref{Tid} show that there is an identifying $s$-set
in $P(q)$ when $s>  \log_2 ((q^2+q)/2)/\log_2(2q/(q+1))=2\log_2(q)-1+o(1)$. The
lower bound $\log_2 (q+1)$ is clear since all the sets $I(v)$, $v\in
V$, must be nonempty and distinct.
\end{proof}

\section{On $\Gr(n,k,\ell)$}\label{Sell}

In this section we consider $\Gr(n,k,\ell)$ for $\ell\ge2$.
Let us denote 
\begin{equation*}
\Xi(k,\ell)=\max\{n : \Gr (n,k,\ell)\neq\emptyset\}.  
\end{equation*}
Trivially, the empty graph $E_k\in\Gr(k,k,\ell)$ for any $\ell\ge1$;
thus $\Xi(k,\ell)\ge $k.

Note that a graph $G=(V,E)$ with $|V|=n$
admits a $(1,\lex\ell)$-identifying set $\iff$ 
$V$ is $(1,\lex\ell)$-identifying 
$\iff$ $G\in\Gr(n,n,\ell)$.

\begin{thm}\label{Tcond}  
Suppose that $G=(V,E)\in \Gr(n,k,\ell)$,
where $n>k$ and $\ell\ge2$.
Then the following conditions hold:
\begin{romenumerate}
 \item For all $x\in V$ we have $\ell+1< n-k+\ell+1\le |N[x]|\le
 k-\ell$. 
In other words, $\gd_G\ge n-k+\ell$ and $\Delta_G\le k-\ell-1$.
 \item For all $x,y\in V$, $x\ne y$, $|N[x]\cap N[y]|\le k-2\ell+1$.
\item
$n\le 2k-2\ell-1$ and $k\ge 2\ell+2$. 
\end{romenumerate}
\end{thm}

\begin{proof}
(i) 
Suppose first that there is a vertex $x\in V$  such that 
$|N[x]|\le n-k+\ell$. By removing $n-k$ vertices from $V$, starting in $\Nx$,
we find a $k$-subset $C$ with $I(C;x)=\{c_1,\ldots, c_{m}\}$
for some $m\le\ell$. 
If $m=0$, then $I(C;x)=I(C;\emptyset)$, which is impossible.
If $1\le m<\ell$, we can arrange (by removing $x$ first)
so that $x\notin C$, and thus
$x\notin Y=\set{c_1,\dots,c_m}$. Then
$I(C;\set{x}\cup Y)=I(C;Y)$, a contradiction.
If $m=\ell\ge2$, we can conversely arrange so that $x\in C$, and thus
$x\in I(C;x)$, say $c_1=x$.
Then $I(C;{c_2,\dots,c_m})=I(C;{c_1,\dots,c_m})$, another
contradiction. Consequently, $|\Nx|\ge n-k+\ell+1$.


Suppose then $|N[x]|\ge k-\ell+1$. 
If $|\Nx|\ge k$, we can choose a $k$-subset $C$ of $\Nx$; then
$I(C;x)=C=I(C;{x,y})$ for any $y$, which is impossible.
If $k>|\Nx|\ge k-\ell+1$, we can choose a $k$-subset
$C=N[x]\cup\{c_1,\ldots c_{k-|N[x]|}\}$.
Choose also 
$a\in N(c_1)$ (which is possible because $\deg(c_1)\ge 1$ by (i)).
Now
$I(C;x,c_1,\ldots,c_{k-|N[x]|})=C=I(C;x,a,c_2,\ldots ,c_{k-|N[x]|})$,
which is impossible. 

(ii) 
Suppose to the contrary that there are
    $x,y\in V$, $x\ne y$, such that $|N[x]\cap N[y]|\ge
    k-2\ell+2$. 
Let $A=N(y)\setminus \Nx$. Then,
according to (i), 
$
|A|
\le|\Ny\setminus\Nx|
=|N[y]|-|N[x]\cap N[y]|
\le k-\ell-(k-2\ell+2)=\ell-2$. 
Since $k>\ell-2$ by (i), there is a $k$-subset $C\subseteq
V\setminus\set y$ such that $A\subset C$.
Then $I(C;A\cup\set{x,y})=I(C;A\cup\set x)$, a contradiction. 

(iii)
An immediate consequence of (i), which implies $n-k+\ell+1\le k-\ell$
and $\ell+1<k-\ell$.
\end{proof}

\begin{thm} \label{GoodL}
For $\ell\geq 2$, $\NN(k,\ell)\le \max\bigset{\frac{\ell}{\ell-1}(k-2),k}$.
\end{thm}
\begin{proof}
If $\,\Xi(k,\ell)=k$, there is nothing to prove.
Assume then that there exists a graph
$G=(V,E)\in \Gr(n,k,\ell)$, where $n> k$. 
By \refT{Tcond}(iii), $\ell<k/2<n$. 
Let us consider any set of
vertices $Z=\{z_1,z_2,\dots,z_\ell\}$ of size $\ell$. We will
estimate $|N[Z]|$ as follows. By \refT{Tcond}(i)
we know $|N[z_1]|\ge
n-k+\ell+1$. Now $N[z_1,z_2]$ must contain at least $n-k+1$
vertices, which \emph{do not} belong to $N[z_1]$ due to
\refT{karakterisointi} which says that $|N[X]\symdiff N[Y]|\ge
  n-k+1$, where we take $X=\{z_1\}$ and $Y=\{z_1,z_2\}$. 
Analogously, each set $N[z_1,\dots, z_i]$ $(i=2,\dots,\ell)$ must
contain at least $n-k+1$ vertices which are not in
$N[z_1,\dots,z_{i-1}]$. Hence, for the set $Z$ we have 
$|N[Z]|\ge n-k+\ell+1+(\ell-1)(n-k+1)=\ell (n-k+2) $. 
Since trivially $|N[Z]|\leq n$, 
we have $(\ell-1)n\le \ell(k-2)$, and 
the claim follows.
\end{proof}

\begin{cor}For $\ell\ge 2$, we have
$
\frac{\Xi(k,\ell)}{k}\leq
1+\frac{1}{\ell-1}$.
\end{cor}

The next results improve the result of Theorem~\ref{GoodL} for
$\ell=2$.

\begin{lemma} \label{Ll2}
Assume that $n>k$. 
Let $G=(V,E)$ belong to $\Gr (n,k,2)$. Then
\[n+\frac{n-k+2}{n-1}(n-k+3)\le 2k-3\]
\end{lemma}
\begin{proof}
Suppose $x\in V$. Let
\[f(n,k)=\frac{n-k+2}{n-1}(n-k+3).\]
Our aim is first to show that there exists a vertex  in $N(x)$ or in
$S_2(x)$ which dominates at least $f(n,k)$ vertices of $\Nx$.
Let 
\[\lambda_x=\max\{|\Nx\cap N[a]|\mid a\in N(x)\}.\]
If $\lambda_x\geq f(n,k)$, we are already done. But if
$\lambda_x<f(n,k)$, then we show that there is a vertex in
$S_2(x)$  that dominates at least $f(n,k)$ vertices of
$\Nx$. Let us estimate the number of edges between the vertices in
$N(x)$ and in $S_2(x)$
--- we denote this number by $M$. 
By \refT{Tcond}(i),
every vertex $y\in N(x)$ yields at
least $|\Ny|-\lambda_x\geq n-k+3-\lambda_x$ such edges and
there are at least $n-k+2$ vertices in $N(x)$. Consequently,
$M\geq (n-k+2)(n-k+3-\lambda_x)$. 
On the other hand, again by \refT{Tcond}(i), $|S_2(x)|\le n-|\Nx|\le k-3$.
Hence, there must exist a vertex in $S_2(x)$ 
incident with at least $M/(k-3)$ edges whose other endpoint is in $N(x)$.
Now, if $\gl_x<f(n,k)$, then
\begin{equation*}
  \frac{M}{k-3} > 
\frac{(n-k+2)(n-k+3-f(n,k))}{k-3}=f(n,k).
\end{equation*}
Hence there exists in this case a vertex
in $S_2(x)$ that is incident to at least $f(n,k)$ such edges, i.e., it
dominates at least $f(n,k)$ vertices in $N(x)$.

In any case there thus exists $z\neq x$ such that $|\Nx\cap \Nz|\ge
f(n,k)$. Let $C=(\Nx\cap\Nz)\cup(V\setminus\Nx)$.
Then $I(C;x,z)=I(C;z)$, so $C$ is not $(1,\lex2)$-identifying and thus
$|C|<k$.
Hence, using \refT{Tcond}(i),
\begin{equation*}
  k-1\ge|C|\ge f(n,k)+n-|\Nx|\ge f(n,k)+n-(k-2),
\end{equation*}
and thus $n+f(n,k)\le 2k-3$ as asserted.
\end{proof}

\begin{thm}\label{Tl2}
  If $k\le 5$, then $\Xi(k,2)=k$. If $k\ge6$, then
$$\Xi(k,2)< \Bigpar{1+\frac{1}{\sqrt{2}}}(k-2)+\frac14.$$
\end{thm}

\begin{proof}
  Let $n=\NN(k,2)$, and let $m=k-2$. If $n>k$, then $k\ge6$ by
  \refT{Tcond}(iii); hence $n=k$ when $k\le5$.
Further, still assuming $n>k$, \refL{Ll2} yields
\begin{equation*}
  n+\frac{(n-m)(n-m+1)}{n-1}\le 2m+1
\end{equation*}
or
\begin{equation*}
  0\ge n(n-1)+(n-m)^2+n-m-(2m+1)(n-1)
=
2\bigpar{n-(m+\tfrac14)}^2-m^2+\tfrac78.
\end{equation*}
Hence, $n-(m+\frac14)<m/\sqrt2$.
\end{proof}

\begin{cor}
For $\ell=2$, we have
$ \Xi(k,2)/k\leq 1+\frac{1}{\sqrt{2}}$.
\end{cor}

\begin{problem}
  What is $\limsup_{k\rightarrow \infty}\Xi(k,\ell)/k$ for $\ell\ge2$?
In particular, is $\limsup_{k\rightarrow \infty}\Xi(k,\ell)/k>1$?
\end{problem}

The following theorem implies that for any $\ell\ge2$ there exist graphs in
$\Gr(n,k,\ell)$ for $n\approx k +\log_2 k$. In
particular, we have such graphs with $n>k$.

\begin{thm}
 Let $\ell\ge 2$ and $m\ge \max\{2\ell-2,4\}$. A binary hypercube of
 dimension $m$ belongs to $\Gr(2^m,2^m-m+2\ell-2,\ell)$ 
\end{thm}
\begin{proof} Suppose first $\ell\ge 3$.
By \cite{laihonen} we know that then a set in a binary hypercube is
$(1,\lex\ell)$-identifying if and only if every vertex is dominated
by at least $2\ell-1$ different vertices belonging to the set.
Hence, we can remove any $m+1-(2\ell-1)$ vertices from the graph, and
there will still be a big enough multiple domination to assure that
the remaining set is $(1,\lex\ell)$-identifying.

Suppose then that $\ell=2$ and $G=(V,E)$ is the binary $m$-dimensional
hypercube.  Let us denote by $C\subseteq V$ a $(2^m-m+2)-$subset.
Every vertex is dominated by at least $m+1-(m-2)=3$ vertices of $C$.
For all $x,y\in V$, $x\ne y$ we have $|N[x]\cap N[y]|=2$ if and only
if $1\le d(x,y)\le 2$ and otherwise $|N[x]\cap N[y]|=0$. Hence,
for all $x,y,z\in V$ with $x\ne y$, $I(y)=\Ny\cap C$ contains at least 3
vertices, and these cannot all be dominated by $x$; thus,
we have $I(x)\ne I(y)$ and $I(x)\ne I(y,z)$. 

We still need to show that
$I(x,y)\ne I(z,w)$ for all $x,y,z,w\in V$, $x\ne y$, $z\ne w$,
$\{x,y\}\ne \{z,w\}$. 
By symmetry we may assume that $x\not\in\{z,w\}$. Suppose $I(x,y)=I(z,w)$. 

If $|I(x)|\ge 5$, then any two vertices $z,w\ne x$ cannot dominate
$I(x)$, a contradiction.

If $|I(x)|=4$, then $|I(z)\cap I(x)|=|I(w)\cap I(x)|=2$ and $I(x)\cap I(z)\cap
I(w)=\emptyset$. It follows that $3\le d(z,w)\le 4$ which implies $I(z)\cap
I(w)=\emptyset$. 
Since $|\Nx\setminus C|=|\Nx|-|I(x)|=m-3$, 
all except one vertex, say $v$, of $V\setminus C$ belong to
$N[x]$, so $V\setminus\Nx\subseteq C\cup\set v$; the vertex $v$ cannot
belong to both $\Nz$ and $\Nw$ since 
these are disjoint, so we may (w.l.o.g.)\ assume that $v\notin\Nz$,
and thus 
$\Nz\setminus\Nx\subseteq C$, whence
$\Nz\setminus\Nx\subseteq I(z)\setminus I(x)$.
Hence,
$|I(z)\cap I(y)|\ge |I(z)\setminus I(x)|\ge|\Nz\setminus\Nx|
=|\Nz|-|\Nz\cap\Nx|=m+1-2\ge 3$. 
Thus  $y=z$; however, then $I(y)\cap I(w)=I(z)\cap I(w)=\emptyset$ and
since $I(w)\not\subseteq I(x)$, we have $I(w)\not\subseteq I(x,y)$.

Suppose finally that $|I(x)|=3$; w.l.o.g.\ we may assume $|I(z)\cap I(x)|=2$.
Now 
$|\Nx\setminus C|=|\Nx|-|I(x)|=m-2=|V\setminus C|$, 
and thus
$V\setminus C=\Nx\setminus C\subseteq N[x]$;
hence, $V\setminus N[x]\subseteq C$ and thus
$\Nz\setminus \Nx\subseteq I(z)\setminus I(x)$.
Consequently, 
$|I(z)\cap I(y)|\ge |I(z)\setminus I(x)|\ge|\Nz\setminus \Nx|\ge m+1-2\ge 3$,
and thus $z=y$. But 
similarly $\Nw\setminus \Nx\subseteq I(w)\setminus I(x)$ and
the same argument shows $w=y$, and thus $w=z$, a
contradiction.
\end{proof}

We finally consider graphs without isolated vertices
(i.e., no vertices with degree zero), and in particular connected graphs.

By \cite[Theorem 8]{lr} a graph with no isolated vertices 
admitting a $(1,\lex\ell)$-identifying set
has minimum degree at least $\ell$. 
Hence, always $n\ge \ell+1$. 

In \cite{GM:ConsIdSets} and
\cite{L:cage} it has been proven that there exist connected graphs which admit
$(1,\lex\ell)$-identifying set. For example, the smallest known connected graph
admitting a $(1,\lex 3)$-identifying set has 16 vertices 
\cite{L:cage}. It is unknown whether there are such graphs with smaller order.
In the next theorem we solve the case of graphs admitting $(1,\lex
2)$-identifying sets.

\begin{thm} 
The smallest $n\ge2$ such that there exists a connected graph (or a
graph without isolated vertices) in $\Gr(n,n,2)$ is $n=7$.
\end{thm}

(If we allow isolated vertices, we can trivially take the empty graph
$E_n$ for any $n\ge2$.)

\begin{proof} 
The cycle $C_n\in\Gr(n,n,2)$ for $n\ge7$ by \refE{exa}\ref{exaC}
(see also \cite{L:cage}).

Assume that $G=(V,E)\in\Gr(n,n,2)$ is a graph of order $n\leq 6$ without
isolated vertices; we will show that this leads to a contradiction. 
By \cite{lr}, we know that  $\deg(v)\ge 2$ for all
$v\in V$. We will use this fact frequently in the sequel.

If $G$ is disconnected, the only possibility is that $n=6$ and that
$G$ consists of
two disjoint triangles, but this graph is not even in $\Gr(n,n,1)$.

Hence, $G$ is connected.
Let $x,y\in V$ be such that $d(x,y)=\diam(G)$.

(i)
Suppose that $\diam(G) =1$, or more generally that there exists a
dominating vertex $x$. Then $N[x,y]=N[x]$ for any $y\in V$, which 
is a contradiction.

(ii)
Suppose next $\diam(G) =2$. 
Moreover, by the previous case we can
assume that for any $v\in V$ there is $w\in V$ such that $d(v,w)=2$.

Assume first  $|N(x)|= 4$. Then $S_2(x)=\set y$. Since $\deg(y)\ge2$,
there exist two vertices $w_1,w_2\in N(y)\cap N(x)$, but then
$N[x,w_1]=N[x,w_2]$.

Assume next $|N(x)|=3$, say $N(x)=\{u_1,u_2,u_3\}$. 
Then $|S_2(x)|=n-|\Nx|\le 2$.
Since the four sets
$N[x]$ and $N[x,u_i]$, $i=1,2,3$, must be distinct, we can assume
without loss of generality that $|S_2(x)|=2$, say $S_2(x)=\{y,w\}$,
and that the only edges between the elements in $S_2(x)$ and $N(x)$
are $u_1y$, $u_2w$, $u_3y$ and $u_3w$. 
Then $N[x,u_3]=N[y,u_2]$.

Assume finally that $|N(x)|=2$. By the previous discussion we may
assume that $|N(v)|=2$ for all $v\in V$. Then $G$ must be a cycle $C_n$,
but it can easily be seen that $C_n\notin\Gr(n,n,2)$ for $3\le n\le6$.

(iii)
Suppose that $\diam(G)=3$. Clearly $|N(x)|\ge 2$ and $|S_2(x)|\geq1$. 
If $|S_2(x)|=1$, say $S_2(x)=\{w\}$, then $N[w,y]=N[w]$, which is not
allowed.
Since $n\le6$, we thus have $|N(x)|=2$ and $|S_2(x)|=2$, say
$N(x)=\{u_1,u_2\}$ and $S_2(x)=\{w_1,w_2\}$.
We can assume without loss of generality that $u_1w_1\in E$.
If $w_2u_2\in E$, then $N[w_1,u_2]=N[x,y]$. 
If $w_2u_2\notin E$, then $N[w_1,w_2]=N[w_1]$.

(iv)
Suppose that $\diam(x,y)\ge4$.
Then $G$ contains an induced path $P_5$. There is 
at most one additional vertex, but it is impossible to add it to $P_5$
and obtain $\gd_G\ge2$ and $\diam(G)\ge4$.

 This completes the proof.
\end{proof}

\begin{ack}
Part of this research was done during the
Workshop on Codes and Discrete Probability in Grenoble, France, 2007.
\end{ack}

\newcommand\xx{\hskip4em}
\newcommand\xxx{\xx\phantom0}
\newcommand\xxxx{\xx\phantom{00}}
\begin{table}[h]
\caption{Lower and upper bounds for $\NN(k)$ for some  $k$.
The lower bounds come from the examples given in the last column; for
$n\ge 8$ using Theorem \ref{TSRG}, \ref{paley} or \ref{TRSHCD} or \refL{Ladd1}. The
strongly regular graphs used here can be found from \cite{CombDesI}. 
The upper bounds for $k\ge 7$ come from Theorem~\ref{kbound}.}
\medskip
\label{tab}
\footnotesize{
\begin{tabular}{rrll}
k& lower bound\hskip-2em{} &\hskip2em  upper bound\hskip-2em{} & example\\
\hline
 1 & 1 & \xxxx1 (Ex.\ \ref{exgr1})& $E_1$\\
 2 & 2 & \xxxx 2 (Ex.\ \ref{exgr2}) & $E_2$\\
 3 & 4 & \xxxx4 (Ex.\ \ref{exgr3}, Th.\ref{kbound}) & $C_4, S_4$\\
 4 & 5 & \xxxx5 (Th.\ \ref{exgr54})& Figure~\ref{Figgr54} \\
 5 & 8 & \xxxx8 (Th.\ \ref{kbound}) & \refE{Ecube}\\
 6 & 9 &\xxxx 9 (Th.\ \ref{upperbound})& \refE{Ebcc}, $P(9)$\\
 7 & 11 & \xxx12 (Th.\ \ref{kbound}, Th.\ \ref{upperbound})& Figure~\ref{figGraph}\\
 8 & 13 & \xxx14 & $P(13)$\\ 
 9 & 16 & \xxx16 & \RSHCD+\\ 
 10 & 17 &\xxx 18 & $P(17)$\\ 
 11 & 18 &\xxx 20 & Th.\ \ref{paley}(ii)\\
 12 & 21 &\xxx 22 & (21,10,3,6)\SRG\\
 13 & 22 &\xxx 24 & \refL{Ladd1} \\ 
 14 & 25 &\xxx 26 & $P(25)$\\ 
 15 & 26 &\xxx 28 & (26,15,8,9)\SRG\\
 16 & 29 &\xxx 30 & $P(29)$\\ 
 17 & 30 &\xxx 32 & Th.\ \ref{paley}(ii)\\
 18 & 31 &\xxx 34 & Th.\ \ref{paley}(ii)\\
 19 & 36 &\xxx 36 & \RSHCD+\\
 20 & 37 &\xxx 38 & $P(37)$\\  
 33 & 64 &\xxx 64 & \RSHCD+\\
  51 & 100 &\xx 100 & \RSHCD+\\
  73 & 144 &\xx 144 & \RSHCD+\\
  99 & 196 &\xx 196 &  \RSHCD+\\
 129 & 256 &\xx 256 & \RSHCD+\\
\end{tabular}}
\end{table}

\bibliographystyle{abbrv}

\end{document}